\documentclass[12pt, reqno]{amsart}

\usepackage{amsmath,amssymb,amsrefs}
\usepackage{hyperref}

\theoremstyle{plain}
\newtheorem{thm}{Theorem}[section]
\newtheorem{prp}[thm]{Proposition}
\newtheorem{lem}[thm]{Lemma}
\newtheorem{cor}[thm]{Corollary}

\theoremstyle{remark}

\newcommand{\C}{\mathbb{C}}
\newcommand{\R}{\mathbb{R}}
\newcommand{\Z}{\mathbb{Z}}
\newcommand{\N}{\mathbb{N}}
\newcommand{\vol}{\operatorname{vol}}
\newcommand{\inrad}{\operatorname{inrad}}

\newcommand{\dist}{\operatorname{dist}}

\newcommand{\beq}{\begin{equation}}
\newcommand{\eeq}{\end{equation}}

\begin{document}

\title[On the inner radius of the nonvanishing set]{On the inner radius of the nonvanishing set for eigenfunctions of complex elliptic operators}

\author{Omer Friedland}
\address{Institut de Math\'ematiques de Jussieu, Sorbonne Universit\'e, 4 place Jussieu, 75005 Paris}
\email{omer.friedland@imj-prg.fr}

\author{Henrik Uebersch\"ar}
\address{Institut de Math\'ematiques de Jussieu, Sorbonne Universit\'e, 4 place Jussieu, 75005 Paris}
\email{henrik.ueberschar@imj-prg.fr}

\date{\today}

\begin{abstract}
Let $\Omega\subset\R^d$ be any open set.
We consider solutions of $H\psi_\lambda=\lambda \psi_\lambda$, $\lambda\in\C$, where $H$ is an $m$th order complex constant-coefficient elliptic partial differential operator. We prove that either the eigenfunctions satisfy a lower bound on the inner radius of the complement of the zero set of $\psi_\lambda$ in $\Omega$ of order $|\lambda|^{-1/m}$, or 100\% of the $L^2$ mass of $\psi_\lambda$ concentrates in a boundary layer of width $|\lambda|^{-1/m}$, as $|\lambda|\to+\infty$.
\end{abstract}

\maketitle

\section{Introduction}

\subsection{Motivation and context}

A classical theme in nodal geometry, going back to Courant, is to understand how the zero set of Laplace eigenfunctions partitions a domain or a manifold. Besides counting nodal domains, a central geometric quantity is the \emph{inner radius} of a nodal domain, i.e.\ the radius of the largest Euclidean ball that can be inscribed in the domain.
For Laplace--Beltrami eigenfunctions on closed manifolds, Mangoubi obtained general lower bounds for the inner radius of \emph{each} nodal domain, with the optimal wavelength scale $\lambda^{-1/2}$ in dimension $2$ and polynomially weaker bounds in higher dimensions \cite{Man08a,Man08b}. More recently, using techniques influenced by Logunov--Malinnikova's work, Charron--Mangoubi proved that in dimensions $d\ge 3$ every nodal domain contains a ball of radius $c\,\lambda^{-1/2}(\log\lambda)^{-(d-2)/2}$, centered at a point of maximal amplitude in that nodal domain \cite{ChMan24}.

A complementary theme relates \emph{mass distribution} of an eigenfunction to geometric lower bounds on inscribed balls. In particular, Georgiev \cite{Geo19} discusses how an $L^2$--distribution condition on coverings by ``good cubes'' forces a nodal domain to have large inradius, and Hezari \cite{Hez18} shows that Mangoubi--type inradius bounds improve along quantum ergodic sequences under small--scale equidistribution hypotheses.

The present paper is of a different nature. We work on an arbitrary open set $\Omega\subset\R^d$ (no boundary regularity is assumed) and consider \emph{constant--coefficient elliptic operators} $H$ of order $m$, allowing complex coefficients and complex spectral parameter $\lambda$. For complex-valued solutions the classical decomposition into nodal domains may be degenerate (the complement of the zero set can be connected), so we focus on the inner radius of the \emph{nonvanishing set} $\Sigma_\lambda=\{\psi_\lambda\neq 0\}$.

Our main result is a quantitative inequality relating $\inrad(\Sigma_\lambda)$ to the proportion of $L^2$--mass carried by the interior $\Omega_{-r_\lambda}$ at the scale $r_\lambda\asymp|\lambda|^{-1/m}$.
We also prove a localized form valid on any open subset $A\subset\Omega$.
Finally, we give an application showing that we have a uniform lower bound $\inrad(\Sigma_\lambda)\gtrsim r_\lambda$, unless 100\% of the $L^2$-mass of $\psi_\lambda$ concentrates in the boundary layer $\Omega\setminus\Omega_{-r_\lambda}$, as $|\lambda|\to+\infty$.

\subsection{Setting and notation}

Let $\Omega\subset\R^d$ be open and let
$$
H = \sum_{|\alpha| = m} c_\alpha D^\alpha,\quad c_\alpha\in\C,
$$
be a constant--coefficient differential operator of order $m\in\N$ whose principal part is homogeneous of degree $m$.
We use the standard notation
$$
D^\alpha = (i\partial_{x_1})^{\alpha_1}\cdots(i\partial_{x_d})^{\alpha_d},
\quad |\alpha| = \alpha_1+\cdots+\alpha_d.
$$
Let
$$
P(\xi) := \sum_{|\alpha| = m} c_\alpha \xi^\alpha
$$
be the (principal) symbol.
We assume $H$ is \emph{elliptic}: there exists $c_{\rm ell}>0$ such that
\begin{equation}\label{eq:ellipticity}
|P(\xi)|\ \ge\ c_{\rm ell} |\xi|^m,\quad \xi\in\R^d.
\end{equation}

Let $\lambda\in\C$ with $|\lambda|\ge 1$ and let $\psi_\lambda\in L^2(\Omega)$ be a distributional solution of
\begin{equation}\label{eq:eigen}
H\psi_\lambda = \lambda\psi_\lambda\quad\text{in }\Omega.
\end{equation}
By interior elliptic regularity, $\psi_\lambda$ is smooth in $\Omega$.

Define the nonvanishing set
$$
\Sigma_\lambda := \{x\in\Omega:\psi_\lambda(x)\neq 0\}.
$$
Since $\psi_\lambda$ is continuous, $\Sigma_\lambda$ is open.
We define its inner radius by
$$
\inrad(\Sigma_\lambda) := \sup\{ \rho>0:\exists x\in\Sigma_\lambda\ \text{with}\ B(x,\rho)\subset \Sigma_\lambda \}.
$$

For $r>0$ we define the $r$--interior of $\Omega$ by
$$
\Omega_{-r} := \{x\in\Omega:\ B(x,r)\subset \Omega\}.
$$
For an open set $A\subset\Omega$ we write
$$
A_{-r} := \{x\in A:\ B(x,r)\subset A\}.
$$
Finally, we set
$$
r_\lambda := |\lambda|^{-1/m}.
$$

\subsection{Main results}

Our main theorem is the following quantitative inradius estimate.
We state it in a scale--invariant form (no global normalization is required).

\begin{thm}[Quantitative inradius bound]\label{thm:main}
Let $\Omega\subset\R^d$ be open and let $H$ be a homogeneous constant--coefficient elliptic operator of order $m$ satisfying \eqref{eq:ellipticity}.
Then there exists a constant $c_{d,H}>0$ (depending only on $d$ and $H$) such that for every $|\lambda|\ge 1$ and every nonzero solution $\psi_\lambda\in L^2(\Omega)$ of \eqref{eq:eigen} we have
\begin{equation}\label{eq:main}
\inrad(\Sigma_\lambda)\ \ge\ c_{d,H} r_\lambda 
\frac{\|\psi_\lambda\|_{L^2(\Omega_{-r_\lambda})}}{\|\psi_\lambda\|_{L^2(\Omega)}}.
\end{equation}
\end{thm}

The following localized version is obtained by applying Theorem~\ref{thm:main} on subsets.

\begin{thm}[Localized inradius bound]\label{thm:local}
Let $\Omega$, $H$ and $\lambda$ be as in Theorem~\ref{thm:main}.
Let $A\subset\Omega$ be open and assume $\|\psi_\lambda\|_{L^2(A)}>0$.
Then
\begin{equation}\label{eq:local}
\inrad(\Sigma_\lambda\cap A)\ \ge\ c_{d,H} r_\lambda 
\frac{\|\psi_\lambda\|_{L^2(A_{-r_\lambda})}}{\|\psi_\lambda\|_{L^2(A)}}.
\end{equation}
\end{thm}

A simple consequence of Theorem~\ref{thm:main}, which is the subject of the following corollary, is that we either have a uniform lower bound $\inrad(\Sigma_{\lambda})\gtrsim r_\lambda$, as $|\lambda|\to+\infty$, or 100\% of the $L^2$-mass of $\psi_\lambda$ concentrates in the boundary layer $\Omega\setminus\Omega_{-r_\lambda}$.
\begin{cor}[Boundary layer concentration]\label{cor:boundary-layer}
Let $(\lambda_j)_{j\in\N}\subset\C$ with $|\lambda_j|\to\infty$ and let $\psi_{\lambda_j}\in L^2(\Omega)$ solve \eqref{eq:eigen} with $\|\psi_{\lambda_j}\|_{L^2(\Omega)} = 1$.
If
$$
\inrad(\Sigma_{\lambda_j}) = o(r_{\lambda_j}),\quad \text{as}\;j\to\infty,
$$
then
$$
\|\psi_{\lambda_j}\|_{L^2(\Omega_{-r_{\lambda_j}})}\ \longrightarrow\ 0,
$$
or, equivalently, 100\% of the $L^2$-mass concentrates in the boundary layer
$\Omega\setminus\Omega_{-r_{\lambda_j}}$ of thickness $\asymp r_{\lambda_j}$.
\end{cor}

\section{Two elementary geometric lemmas}

We record two simple lemmas that convert Lipschitz control and $L^2$--mass into an inscribed nonvanishing ball.

\subsection{A nonvanishing ball from Lipschitz control}

\begin{lem}[Lipschitz nonvanishing ball]\label{lem:lipschitz-ball}
Let $\Omega\subset\R^d$ be open and let $G:\Omega\to\C$ be Lipschitz on $\Omega$ with Lipschitz constant $L$, i.e.
$|G(x)-G(y)|\le L|x-y|$ for all $x,y\in\Omega$.
Fix $x_0\in\Omega$ and assume $|G(x_0)|\ge \eta>0$.
Set
$$
\rho := \min\Bigl\{\frac{\eta}{2L},\ \dist(x_0,\Omega^c)\Bigr\}.
$$
Then
$$
B(x_0,\rho)\ \subset\ \{x\in\Omega:\ G(x)\neq 0\}.
$$
\end{lem}

\begin{proof}
By definition of $\rho$, we have $B(x_0,\rho)\subset\Omega$.
For $x\in B(x_0,\rho)$ we estimate
$$
|G(x)|\ \ge\ |G(x_0)|-|G(x)-G(x_0)|
\ \ge\ \eta-L|x-x_0|
\ \ge\ \eta-L\rho
\ \ge\ \eta-\frac{\eta}{2}\ = \ \frac{\eta}{2}>0.
$$
\end{proof}

\subsection{Finding a point of large amplitude from an $L^2$ lower bound}

\begin{lem}[Pointwise lower bound from $L^2$ mass]\label{lem:L2-to-Linfty}
Let $u\in L^2(B(0,R))$. Then
$$
\sup_{B(0,R)} |u|\ \ge\ \frac{\|u\|_{L^2(B(0,R))}}{\vol(B(0,R))^{1/2}}.
$$
\end{lem}

\begin{proof}
Since $|u|^2\le \big(\sup_{B(0,R)}|u|\big)^2$ almost everywhere on $B(0,R)$,
$$
\|u\|_{L^2(B(0,R))}^2 = \int_{B(0,R)}|u|^2
\le \vol(B(0,R)) \Big(\sup_{B(0,R)}|u|\Big)^2.
$$
Taking square roots yields the claim.
\end{proof}

\section{A uniform local Lipschitz bound at bounded spectral parameter}

In the proof of Theorem~\ref{thm:main} we need a uniform interior $C^1$ bound for solutions of
$Hu = \mu u$ in a fixed ball, with $\mu$ ranging over a compact set.
We obtain this by invoking a local derivative estimate from \cite{FU24}.
Throughout this section we fix once and for all
\begin{equation}\label{eq:delta-fixed}
\delta := \frac12.
\end{equation}

\subsection{A local derivative estimate from \cite{FU24}}

\begin{thm}[Local derivative bound {\cite[Theorem~1.1]{FU24}}]\label{thm:FU24}
Let $H = \sum_{|\alpha|\le m}c_\alpha D^\alpha$ be a constant--coefficient elliptic operator of order $m$ on $\R^d$, with symbol $P(\xi) = \sum_{|\alpha|\le m}c_\alpha\xi^\alpha$.
Fix $\delta$ by \eqref{eq:delta-fixed}.
Let $x\in\Omega$, $r\in(0,1)$, and assume $B(x,r)\subset\Omega$.
If $\psi$ solves $H\psi = \lambda\psi$ in $\Omega$ with $|\lambda|\ge 1$, then for every multi-index $\gamma$,
$$
|D^\gamma \psi(x)|
\ \le\ C_{d,r,\gamma,H} |\lambda|^{|\gamma|/m}\Bigl(1+\mathcal N_\lambda^{1/2}\Bigr) 
\|\psi\|_{L^2(B(x,r))},
$$
where
$$
\mathcal N_\lambda := \#\Bigl\{\xi\in\Z^d:\ |P(\xi)-\lambda|\le |\xi|^{m-1+\delta}\Bigr\}.
$$
\end{thm}

\begin{lem}[Uniform boundedness of $\mathcal N_\lambda$ on compact sets]\label{lem:Nlambda-compact}
Let $K\subset\C$ be compact. Then
$$
\sup_{\lambda\in K}\mathcal N_\lambda\ <\ \infty.
$$
\end{lem}

\begin{proof}
Write $P(\xi) = P_m(\xi)+P_{m-1}(\xi)$, where $P_m$ is the homogeneous principal part and $P_{m-1}$ has degree at most $m-1$.
By ellipticity of $P_m$ there exists $c_0>0$ such that $|P_m(\xi)|\ge c_0|\xi|^m$ for all $\xi\in\R^d$.
Since $P_{m-1}(\xi) = O(|\xi|^{m-1})$, there exists $R_0\ge 1$ such that
$$
|P(\xi)|\ge \frac{c_0}{2}|\xi|^m\quad \text{for all }|\xi|\ge R_0.
$$
Let $M := \sup_{\lambda\in K}|\lambda|<\infty$.
If $|\xi|\ge R_0$ and $|P(\xi)-\lambda|\le|\xi|^{m-1+\delta}$ for some $\lambda\in K$, then
$$
\frac{c_0}{2}|\xi|^m\ \le\ |P(\xi)|
\ \le\ |\lambda|+|P(\xi)-\lambda|
\ \le\ M+|\xi|^{m-1+\delta}.
$$
Since $\delta = \tfrac12<1$, the right-hand side grows strictly slower than $|\xi|^m$ as $|\xi|\to\infty$, hence this inequality can hold only for $|\xi|\le R_1$ for some $R_1 = R_1(K,H)$.
Therefore, for every $\lambda\in K$, $\mathcal N_\lambda$ counts lattice points in the finite set $\{\xi\in\Z^d:\ |\xi|\le R_1\}$, so $\mathcal N_\lambda$ is uniformly bounded on $K$.
\end{proof}

\subsection{A uniform $C^1$ bound on $B(0,3/4)$}

\begin{prp}[Uniform local Lipschitz bound]\label{prp:uniform-lip}
Let $H$ be as in Theorem~\ref{thm:main}.
Then there exists a constant $L_{d,H}>0$ such that for every $\mu\in\C$ with $1\le|\mu|\le 2$
and every solution $u\in L^2(B(0,1))$ of
$$
Hu = \mu u\quad\text{in }B(0,1),
$$
we have
$$
\sup_{x\in B(0,3/4)}|\nabla u(x)|\ \le\ L_{d,H} \|u\|_{L^2(B(0,1))}.
$$
In particular, $u$ is Lipschitz on $B(0,3/4)$ with Lipschitz constant at most
$L_{d,H}\|u\|_{L^2(B(0,1))}$.
\end{prp}

\begin{proof}
Fix $x\in B(0,3/4)$. Then $B(x,1/4)\subset B(0,1)$.
Apply Theorem~\ref{thm:FU24} with $r = 1/4$, $\lambda = \mu$, and with $\gamma$ any first-order multi-index.
Since $1\le|\mu|\le 2$, we obtain
$$
|\nabla u(x)|\ \le\ C_{d,1/4,1,H} |\mu|^{1/m}\bigl(1+\mathcal N_\mu^{1/2}\bigr) \|u\|_{L^2(B(x,1/4))}.
$$
Because $\|u\|_{L^2(B(x,1/4))}\le\|u\|_{L^2(B(0,1))}$ and $|\mu|^{1/m}\le 2^{1/m}$, it remains to bound $\mathcal N_\mu$ uniformly for $1\le|\mu|\le 2$.
This follows from Lemma~\ref{lem:Nlambda-compact} with $K = \{\mu\in\C:\ 1\le|\mu|\le 2\}$.
Absorbing constants gives the claim.
\end{proof}

\section{A bounded overlap cover and a mass ratio lemma}

\begin{lem}[Bounded overlap cover]\label{lem:cover}
Let $E\subset\R^d$ and let $r>0$.
There exists a (finite or countable) set $\{x_j\}_{j\in J}\subset E$ with the following properties:
\begin{enumerate}
\item $E\subset \bigcup_{j\in J} B(x_j,r/2)$.
\item The family of dilated balls $\{B(x_j,r)\}_{j\in J}$ has bounded overlap: there exists $N_d\in\N$ depending only on $d$ such that each point of $\R^d$ belongs to at most $N_d$ balls $B(x_j,r)$.
\end{enumerate}
\end{lem}

\begin{proof}
Choose $\{x_j\}_{j\in J}\subset E$ maximal with respect to the property that the balls $B(x_j,r/4)$ are pairwise disjoint.
Maximality implies $E\subset\bigcup_{j\in J}B(x_j,r/2)$: otherwise one could add a point of $E$ not covered and still preserve disjointness of the $r/4$--balls.

For the overlap bound, fix $x\in\R^d$ and consider the set $J(x) := \{j\in J:\ x\in B(x_j,r)\}$.
Then $x_j\in B(x,r)$ for all $j\in J(x)$, so the disjoint balls $B(x_j,r/4)$, $j\in J(x)$, are all contained in $B(x,5r/4)$.
Comparing volumes yields
$$
\#J(x) \vol(B(0,r/4))\ \le\ \vol(B(0,5r/4)),
$$
hence $\#J(x)\le 5^d$.
Thus one may take $N_d := 5^d$.
\end{proof}

\begin{lem}[Existence of a good ball]\label{lem:good-ball}
Let $r>0$ and set $E := \Omega_{-r}$.
Let $f\in L^1(\Omega)$ satisfy $f\ge 0$ and $\int_\Omega f>0$.
Define $M := \int_E f$.
Let $\{x_j\}_{j\in J}\subset E$ be as in Lemma~\ref{lem:cover}, so that $E\subset \bigcup_{j} B(x_j,r/2)$ and the balls $B(x_j,r)$ have overlap bounded by $N_d$.
Then there exists $j_0\in J$ such that
\begin{equation}\label{eq:ratio}
\frac{\int_{B(x_{j_0},r/2)} f}{\int_{B(x_{j_0},r)} f}
\ \ge\
\frac{M}{2N_d \int_\Omega f}.
\end{equation}
\end{lem}

\begin{proof}
If $M = 0$ the claim is vacuous, so assume $M>0$.
Since $E\subset\bigcup_j B(x_j,r/2)$ we have
$$
M\le \sum_{j\in J}\int_{B(x_j,r/2)} f.
$$
Argue by contradiction and assume that for every $j\in J$,
$$
\int_{B(x_j,r/2)} f\ <\ \alpha \int_{B(x_j,r)} f,
\quad\text{where }\alpha := \frac{M}{2N_d \int_\Omega f}.
$$
Then
$$
M
< \alpha \sum_{j\in J}\int_{B(x_j,r)} f.
$$
By bounded overlap,
$$
\sum_{j\in J}\int_{B(x_j,r)} f\ \le\ N_d\int_\Omega f.
$$
Hence $M<\alpha N_d\int_\Omega f = M/2$, a contradiction. Therefore \eqref{eq:ratio} holds for some $j_0$.
\end{proof}

\section{Proofs of the main results}

\begin{proof}[Proof of Theorem~\ref{thm:main}]
Let $\psi_\lambda\in L^2(\Omega)$ be a nonzero solution of \eqref{eq:eigen}.
Set $r := r_\lambda = |\lambda|^{-1/m}$ and define
$$
M := \|\psi_\lambda\|_{L^2(\Omega_{-r})}^2,\quad
N := \|\psi_\lambda\|_{L^2(\Omega)}^2>0.
$$
If $M = 0$ then \eqref{eq:main} is trivial. Assume $M>0$.

\textbf{Step 1: choose a good ball at scale $r$.}
Apply Lemma~\ref{lem:good-ball} with $f = |\psi_\lambda|^2$, $E = \Omega_{-r}$.
Since $\int_\Omega f = N$, there exists $x_0\in\Omega_{-r}$ such that
\begin{equation}\label{eq:good-ratio}
\frac{\|\psi_\lambda\|_{L^2(B(x_0,r/2))}^2}{\|\psi_\lambda\|_{L^2(B(x_0,r))}^2}
\ \ge\ \frac{M}{2N_d N}.
\end{equation}
Because $x_0\in\Omega_{-r}$, we have $B(x_0,r)\subset\Omega$.

\textbf{Step 2: rescale to unit scale.}
Define the rescaled and normalized function $u:B(0,1)\to\C$ by
\begin{equation}\label{eq:rescale}
u(y) := \frac{r^{d/2} \psi_\lambda(x_0+r y)}{\|\psi_\lambda\|_{L^2(B(x_0,r))}}.
\end{equation}
Then $\|u\|_{L^2(B(0,1))} = 1$, and similarly
\begin{equation}\label{eq:u-half}
\|u\|_{L^2(B(0,1/2))}^2
 = \frac{\|\psi_\lambda\|_{L^2(B(x_0,r/2))}^2}{\|\psi_\lambda\|_{L^2(B(x_0,r))}^2}
\ \ge\ \frac{M}{2N_d N}.
\end{equation}

Since $H$ is homogeneous of degree $m$, the rescaling \eqref{eq:rescale} converts \eqref{eq:eigen} into
\begin{equation}\label{eq:scaled-eqn}
Hu = \mu u\quad\text{in }B(0,1),
\quad \mu := r^m\lambda = \frac{\lambda}{|\lambda|}.
\end{equation}
In particular, $|\mu| = 1$.

\textbf{Step 3: find a point where $|u|$ is large.}
By Lemma~\ref{lem:L2-to-Linfty} with $R = 1/2$ and \eqref{eq:u-half}, there exists $y_0\in B(0,1/2)$ such that
\begin{equation}\label{eq:amp-lower}
|u(y_0)|\ \ge\ \frac{\|u\|_{L^2(B(0,1/2))}}{\vol(B(0,1/2))^{1/2}}
\ \ge\ \frac{1}{\vol(B(0,1/2))^{1/2}}\Bigl(\frac{M}{2N_d N}\Bigr)^{1/2}.
\end{equation}

\textbf{Step 4: use a uniform Lipschitz bound to inscribe a nonvanishing ball.}
Since $|\mu| = 1$, Proposition~\ref{prp:uniform-lip} (with $\|u\|_{L^2(B(0,1))} = 1$) yields
$$
|u(x)-u(y)|\le L_{d,H}|x-y|\quad\text{for all }x,y\in B(0,3/4).
$$
Apply Lemma~\ref{lem:lipschitz-ball} to the open set $\Omega' = B(0,3/4)$, the function $G = u|_{\Omega'}$, and the point $x_0 = y_0\in B(0,1/2)\subset\Omega'$.
Since $\dist(y_0,(\Omega')^c) = \frac34-|y_0|\ge \frac14$, Lemma~\ref{lem:lipschitz-ball} gives a ball
$$
B(y_0,\rho_0)\subset\{u\neq 0\},
\quad
\rho_0 := \min\Bigl\{\frac{|u(y_0)|}{2L_{d,H}},\ \frac14\Bigr\}.
$$
Scaling back via $x = x_0+r y$ yields
$$
B(x_0+r y_0,\ r\rho_0)\subset\{\psi_\lambda\neq 0\} = \Sigma_\lambda,
$$
so $\inrad(\Sigma_\lambda)\ge r\rho_0$.

Since $\sqrt{M/N}\le 1$, the elementary inequality $\min\{a t,b\}\ge \min\{a,b\} t$ for $t\in[0,1]$
together with \eqref{eq:amp-lower} yields
$$
\rho_0
\ge
\min\Bigl\{
\frac{1}{2L_{d,H} \vol(B(0,1/2))^{1/2}}\Bigl(\frac{1}{2N_d}\Bigr)^{1/2},
\ \frac14
\Bigr\} \Bigl(\frac{M}{N}\Bigr)^{1/2}.
$$
Therefore
$$
\inrad(\Sigma_\lambda)
\ge
c_{d,H} r \Bigl(\frac{M}{N}\Bigr)^{1/2}
 = 
c_{d,H} r_\lambda 
\frac{\|\psi_\lambda\|_{L^2(\Omega_{-r_\lambda})}}{\|\psi_\lambda\|_{L^2(\Omega)}},
$$
where
$$
c_{d,H} := 
\min\Bigl\{
\frac{1}{2L_{d,H} \vol(B(0,1/2))^{1/2}}\Bigl(\frac{1}{2N_d}\Bigr)^{1/2},
\ \frac14
\Bigr\}>0.
$$
This proves \eqref{eq:main}.
\end{proof}

\begin{proof}[Proof of Theorem~\ref{thm:local}]
Apply Theorem~\ref{thm:main} to the open set $A$ and to the function $\psi_\lambda|_A$.
Since \eqref{eq:eigen} holds in $\Omega$ in the distributional sense, it holds in $A$ as well.
The nonvanishing set in $A$ is $\Sigma_\lambda\cap A$, and the $r$--interior of $A$ is $A_{-r}$.
Thus Theorem~\ref{thm:main} yields \eqref{eq:local}.
\end{proof}

\begin{proof}[Proof of Corollary~\ref{cor:boundary-layer}]
With $\|\psi_{\lambda_j}\|_{L^2(\Omega)} = 1$, Theorem~\ref{thm:main} gives
$$
\|\psi_{\lambda_j}\|_{L^2(\Omega_{-r_{\lambda_j}})}
\ \le\ \frac{\inrad(\Sigma_{\lambda_j})}{c_{d,H} r_{\lambda_j}}.
$$
If $\inrad(\Sigma_{\lambda_j}) = o(r_{\lambda_j})$, the right-hand side tends to $0$.
\end{proof}

\end{document}